\documentclass[10pt, a4paper]{article}
\usepackage{titlesec} 
\usepackage{amsfonts}
\usepackage{amsthm}
\usepackage{amssymb}
\usepackage{amsmath}
\usepackage{theoremref}
\usepackage{enumerate}
\usepackage{bbm}
\usepackage{authblk}

\newcommand{\A}{\mathcal{A}}
\newcommand{\C}{\mathcal{C}}
\newcommand{\T}{\mathbb{T}}
\newcommand{\z}{\zeta}
\newcommand{\conj}[1]{\overline{#1}}

\newcommand{\D}{\mathbb{D}}

\newcommand{\hil}{\mathcal{H}}
\newcommand{\Hb}{\mathcal{H}(b)}

\newtheorem{thm}{Theorem}
\newtheorem*{thm*}{Theorem}  
\newtheorem*{conj*}{Conjecture}  
\newtheorem{lemma}[thm]{Lemma}

\newtheorem{prop}[thm]{Proposition}
\theoremstyle{definition}

\theoremstyle{definition}

\titleformat{\section}
{\normalfont\scshape\centering}{\thesection}{1em}{}

\titleformat{\subsection}
[runin]{\bfseries\scshape}{\thesubsection}{0.5em}{}

\begin{document}
\title{\textbf{Density of disc algebra functions in de Branges-Rovnyak spaces}}
\author{ Alexandru Aleman and Bartosz Malman}
\date{ }

\maketitle

\newcommand{\Addresses}{{
		\bigskip
		\footnotesize
		
		Alexandru Aleman, \\ \textsc{Centre for Mathematical Sciences, Lund University, \\
		Lund, Sweden}\\
		\texttt{aleman@maths.lth.se}
		
		\medskip
		
		Bartosz Malman, \\ \textsc{Centre for Mathematical Sciences, Lund University, \\
				Lund, Sweden} \\
			\texttt{bartosz.malman@math.lu.se}
		
	}}

\begin{abstract}
	We prove that functions continuous up to the boundary are dense in de Branges-Rovnyak spaces induced by extreme points the unit ball of $H^\infty$.
\end{abstract}
	
\section{Introduction}

Let   $H^\infty$ be  the algebra of bounded analytic functions in the unit disk $\mathbb{D}$ in the complex plane, and denote by $\A$ the disc algebra,  i.e. the subalgebra of $H^\infty$ consisting of functions which extend continuously to the closed disk.
The  Hardy space $H^2$ consists of  power series in $\mathbb{D}$ with square-summable  coefficients. If $\T$ denotes the unit circle, we identify  as usual $H^2$ with the closed subspace of $L^2(\T)$ consisting of functions  whose negative Fourier coefficients vanish. The orthogonal projection from $L^2(\T)$ onto $H^2$ is denoted by $P_+$.
 
For  $\phi \in L^\infty(\T)$ let $T_\phi$ denote the Toeplitz operator on $H^2$  defined  by $T_\phi f = P_+ \phi f$.  Given $b\in H^\infty$ with $\|b\|_\infty \leq 1$ we define the corresponding \textit{de Branges-Rovnyak space} $\Hb$  as $$\Hb = (1 - T_b T_{\conj{b}})^{1/2}H^2.$$ $\Hb$ is endowed with  the unique norm which makes the operator $(1 - T_b T_{\conj{b}})^{1/2}$ a partial isometry from $H^2$ onto $\Hb$. Alternatively, $\Hb$ is defined as the reproducing kernel Hilbert space with  kernel  $$k_b(z,\lambda) = \frac{1-\conj{b(\lambda)}b(z)}{1-\conj{\lambda}z}.$$  

 $\Hb$-spaces are naturally split into  two classes with fairly different structures according to whether the quantity $\int_{\T} \log(1-|b|) \,dm$ is finite or not. Here $m$ denotes the normalized arc-length measure on $\T$. The present note concerns the approximation of $\Hb$-functions  by functions in $\A\cap \Hb$ and from the technical point of view there is a major difference between the two classes, which we shall briefly explain.
 
 If  $\int_{\T} \log(1-|b|) \,dm<\infty$, or equivalently,  if $b$  is a non-extreme point of the unit ball of $H^\infty$ (see \cite{sarasonbook}), then $\Hb$ contains all functions analytic  in a neighborhood of the closed unit disk. By a theorem of Sarason, the polynomials form a norm-dense subset of the space \cite{sarasonbook}. An interesting feature of the proofs of density of polynomials in an $\Hb$-space is that the usual approach of approximating a function $f$ first by its dilations $f_r(z) = f(rz)$, and then by their truncated Taylor series, or by their Ces\`aro means, does not work. Sarason's intial proof of density of polynomials is based on  a duality argument. In recent years a more involved constructive polynomial approximation scheme has been obtained in \cite{constr}.

The picture changes dramatically  in the case when $\int_{\T} \log(1-|b|) \,dm=\infty$, or equivalently when $b$ is an extreme point of the unit ball of $H^\infty$. Then  it is in general a difficult task to identify any functions in the space other than the reproducing kernels, and it might  happen that $\Hb$ contains no non-zero function analytic in a neighborhood of the closed disk. A special class of extreme points are the inner functions. If $b$ is inner then  $\Hb = H^2 \ominus bH^2$ with equality of norms,  and it is a consequence of a celebrated theorem of Aleksandrov  \cite{aleksandrovinv} that in this case the intersection $\A \cap \Hb$ is dense in the space. The result is surprising since, as pointed out above, in most cases it is not obvious at all that $\Hb$ contains any non-zero function in the disk algebra $\A$.

Motivated  by the situation described here, E. Fricain \cite{openprob}, raised the natural question whether  Aleksandrov's result extends to all other $\Hb$-spaces induced by extreme points $b$ of the unit ball of $H^\infty$. It is the purpose of this note to provide an affirmative answer to this question, contained in the main result below.

\begin{thm} \thlabel{theorem}
	If $b$ is  an extreme point of the unit ball of $H^\infty$ then $\A \cap \Hb$ is a dense subset of $\Hb$. 
\end{thm}

\noindent Together with Sarason's result \cite{sarasonbook} on the density of polynomials  in the non-extreme case, it follows that the intersection $\A \cap \Hb$ is  dense in the space $\Hb$ for any $b$ in the unit ball of $H^\infty$. Our proof Theorem \ref{theorem} is deferred to Section 3 and relies on a duality argument. Therefore, just as the earlier proofs of Sarason and Aleksandrov, our approach is non-constructive. Section 2 serves to a  preliminary purpose.

\section{Preliminaries}
\subsection{The norm on $\Hb$.} An essential  step is the following useful representation of the norm in $\Hb$. The authors have originally deduced the result using the techniques in \cite{ars} (see also \cite[Chapter 3]{afr}), but once the  goal is identified, several available  techniques provide simpler proofs. For example,  the
 proposition below can be deduced from results  in \cite{sarasonbook}. For the sake completeness, we include a new  shorter proof.

\begin{prop} \thlabel{normformula}
Let $b$ be an extreme point of the unit ball of $H^\infty$ and let $$E = \{ \zeta \in \T : |b(\zeta)| < 1 \}.$$ Then for $f\in \Hb$
the equation
 $$P_+ \conj{b}f = -P_+ \sqrt{1-|b|^2}g.$$
has a unique solution $g\in L^2(E)$, and the map $J:\Hb\to H^2\oplus L^2(E)$  defined by $$Jf=(f,g),$$ 
is an isometry. Moreover,  $$J(\Hb)^\perp = \Big\{ (bh, \sqrt{1-|b|^2}h) : h \in H^2 \Big\}.$$ 
\end{prop}

\begin{proof}
	Let $$K = \Big\{ (bh, \sqrt{1-|b|^2}h) : h \in H^2 \Big\},$$
and let $P_1$ be the projection from $H^2 \oplus L^2(E)$ onto the first coordinate $H^2$, i.e., $P_1(f,g) = f$. We observe first that $P_1|K^\perp$ is injective. Indeed, if $K^\perp$ contains a tuple of the form $(0,g) \in H^2 \oplus L^2(E)$,  it  follows that 
$$\int_{\T} \conj{\zeta}^ng(\zeta)\sqrt{1-|b(\zeta)|^2} \,dm(\zeta) = 0, \, n \geq 0,$$ and consequently
the function $g\sqrt{1-|b|^2}$ coincides a.e. with the boundary values of the complex conjugate of a function $f\in H^2_0$. But the assumption that $b$ is an extreme point then implies that $\int_\T \log|f| \,dm = -\infty$, and since $f\in H^2$, we conclude that $f = 0$, i.e, $g = 0$. Thus, the space $\hil = P_1K^\perp$  with the norm $\|f\|_\hil = \|P_1^{-1}f\|_{H^2 \oplus L^2(E)}$ is a Hilbert space of analytic functions on $\D$, contractively contained in $H^2$, in particular, it is a reproducing kernel Hilbert space. We now show that $\hil$ equals $\Hb$ by verifying that  the reproducing kernels of the two spaces coincide. This follows from a simple computation. For $\lambda\in \D$, the tuple
 \begin{align*}(f_\lambda,g_\lambda)&=\Big( \frac{1-\conj{b(\lambda)}b(z)}{1-\conj{\lambda}z}, -\frac{\conj{b(\lambda)}\sqrt{1-|b(z)|^2}}{1-\conj{\lambda}z}\Big) \\&= \Big(\frac{1}{1-\conj{\lambda}z},0\Big) - \Big( \frac{\conj{b(\lambda)}b(z)}{1-\conj{\lambda}z}, \frac{\conj{b(\lambda)}\sqrt{1-|b(z)|^2}}{1-\conj{\lambda}z}\Big)\end{align*}
is obviously orthogonal to $K$, while the last tuple on the right hand side is in $K$, so that $f_\lambda$ is the reproducing
kernel in $\hil$, which obviously equals the reproducing kernel in $\Hb$. The first assertion in the statement is now self-explanatory.
\end{proof}

\subsection{The Khintchin-Ostrowski theorem.}  Recall that analytic functions $f$ in $\D$ satisfy $\sup_{0<r<1} \int_\T \log^+|f_r| dm<\infty$ if and only if they are quotients of  $H^\infty$-functions, in particular they have finite nontangential limits a.e. on $\T$ which define a boundary function denoted also by $f$. The class $N^+(\D)$ consists of quotients of $H^\infty$-functions such that the denominator can be chosen to be outer, it contains all Hardy spaces $H^p, ~p>0$. The Khintchin-Ostrowski theorem reads as follows. A proof can be found in \cite{havinbook}.
\begin{thm} \thlabel{ostrowski}
	Let $\{f_n\}$ be a sequence of functions analytic in the unit disk satisfying the following conditions:
	\begin{enumerate}[(i)]
		\item There exists a constant $C > 0$ such that $$\int_\T \log^+(|f_n(re^{it})|) dt \leq C.$$
		\item On some set $E$ of positive measure, the sequence $f_n$ converges in measure to a function $\phi$.		
	\end{enumerate}
	Then the sequence $f_n$ converges uniformly on compact subsets of the unit disk to a function $f$  which satisfies $f=\phi$ a.e. on $E$.
\end{thm}	
\section{Proof of the main result} Due to Proposition \ref{normformula} we can now implement Aleksandrov's strategy from  \cite{aleksandrovinv}  which will then be combined with the  Khintchin-Ostrowski theorem.\\
Recall that the dual $\A'$  of the disk algebra $\A$ can  be identified with the space $\C$ of Cauchy transforms of finite measures on $\T$ (\cite{cauchytransform}) via the pairing $$\langle f,C\mu\rangle=\lim_{r\to 1^-}\int_\T f(\z)\conj{C\mu(r\z)}dm(\z)=\int_\T fd\conj{\mu},$$
where $$C\mu(z)=\int_\T\frac1{1-z\conj{\z}}d\mu(\z)$$
is the Cauchy transform of $\mu$. The space $\C$ is endowed with the obvious quotient norm and is continuously contained in all $H^p$ spaces for $0<p<1$. The following result extends  Alexandrov's approach to the context of $\Hb$-spaces, when $b$ is extremal in the unit ball of $H^\infty$.
	
\begin{lemma} \thlabel{wsclosed}
	 Let $E = \{\z \in \T : |b(\z)| < 1\}$, $B = \mathcal{A} \oplus L^2(E)$ and $B' = C \oplus L^2(E)$. Then the set $$S = \{ (C\mu, h) : C\mu/b \in N^+(\D), C\mu/b = h/\sqrt{1-|b|^2} \mbox{ a.e. on } E \}$$ is weak-* closed in $B'$.
\end{lemma} 

\begin{proof}
	Since $\mathcal{A} \oplus L^2(E)$ is separable, it will be sufficient to show that $S$ is weak-* sequentially closed. Let $(C\mu_n, h_n)$ converge weak-* to $(C\mu, h)$, where $(C\mu_n, h_n) \in S$ for $n \geq 1$. Equivalently,  
 $h_n\to h$   weakly  in $L^2(E)$, and
$$\sup_n \|C\mu_n\| < \infty. \quad \lim_{n \rightarrow \infty} C\mu_n(z) = C\mu(z),\,\,z\in \D.$$  Now by passing to a subsequence and the Ces\`aro means of that subsequence we can assume that $h_n \to  h$ in the $L^2$-norm. Finally using another subsequence we may also  assume that  $h_n\to h$  pointwise a.e. on $E$. Let $I_b$ be the inner factor of $b$. Since $C\mu_n/I_b\in
N^+(\D)$, it follows by  Vinogradov's theorem (\cite[Theorem 6.5.1]{cauchytransform}) that $\{C\mu_n/I_b\}_n$ is a bounded sequence in $\C$ converging pointwise on $\D$ to  $C\mu/I_b$.  This implies  weak-* convergence in $\C$, in particular, $C\mu/I_b \in C \subset N^+(\D)$, and consequently, $C\mu/b \in N^+(\D)$. Moreover,   we have a.e. on $E$ that  $C\mu_n/b = h_n/\sqrt{1-|b|^2}$ which converges pointwise to $h/\sqrt{1-|b|^2}$, hence we conclude that the sequence $C\mu_n$ converges  in measure to some function $\phi$ on $E$. Finally, if $p \in (0,1)$ and $\|\cdot\|_p$ denotes the $H^p$-norm, then \begin{eqnarray*}
		\int_\T \log^+(|C\mu_n(re^{it})|) dt  \lesssim \int_\T |C\mu_n(re^{it})|^p dt \lesssim \|C\mu_n\|^p_p \lesssim \sup_n \|C\mu_n\| < \infty.
	\end{eqnarray*} Thus the assumptions of \thref{ostrowski} are  satisfied, and so (a subsequence of)  $C\mu_n$ converges a.e.  on $E$ to $C\mu$. This clearly implies $C\mu/b=h/\sqrt{1-|b|^2}$ a.e. on $E$, i.e.
$(C\mu, h) \in S$.
\end{proof}

We are now ready to complete the proof of the main theorem.

\begin{proof}[Proof of \thref{theorem}]
	Let $J$ denote the embedding in \thref{normformula}. Based on the pairing described at the beginning of this section,  a direct application of  \thref{normformula}  gives
 $$J(\A \cap \Hb) = \cap_{h \in H^2} \ker l_h,$$ where the functionals $l_h$ are identified with elements of $\C \oplus L^2(E)$ as
	$$l_h = \Big(hb, h\sqrt{1-|b|^2}\Big).$$
It is a consequence of the Hahn-Banach theorem that the annihilator $J(\A \cap \Hb)^\perp$ is the weak-* closure of the set of the functionals $l_h$. Since  for all $h \in H^2$ we have $l_h \in S$, the set considered in  \thref{wsclosed},  by the lemma we conclude  that  $J(\A \cap \Hb)^\perp \subset S$. Thus if $f\in \Hb$ is orthogonal to  $\A \cap \Hb$, we must have  $Jf \in S$, that is  $$Jf = (hb, h\sqrt{1-|b|^2})$$ for some $h \in H^2$.  But then by \thref{normformula}, $Jf \in J(\Hb)^\perp$, which gives $Jf = 0$ and the proof is complete.
\end{proof}

\bibliographystyle{siam}
\bibliography{mybib}

\begin{thebibliography}{1}

\bibitem{aleksandrovinv}
{\sc A.~B. Aleksandrov}, {\em Invariant subspaces of shift operators. {A}n
  axiomatic approach}, Zap. Nauchn. Sem. Leningrad. Otdel. Mat. Inst. Steklov.
  (LOMI), 113 (1981), pp.~7--26, 264.

\bibitem{afr}
{\sc A.~Aleman, N.~S. Feldman, and W.~T. Ross}, {\em The {H}ardy space of a
  slit domain}, Frontiers in Mathematics, Birkh\"auser Verlag, Basel, 2009.

\bibitem{ars}
{\sc A.~Aleman, S.~Richter, and C.~Sundberg}, {\em Beurling's theorem for the
  {B}ergman space}, Acta Math., 177 (1996), pp.~275--310.

\bibitem{openprob}
{\sc C.~B\'en\'eteau, A.~Condori, C.~Liaw, W.~T. Ross, and A.~Sola}, {\em Some
  open problems in complex and harmonic analysis: report on problem session
  held during the conference {C}ompleteness problems, {C}arleson measures, and
  spaces of analytic functions}, in Recent progress on operator theory and
  approximation in spaces of analytic functions, vol.~679 of Contemp. Math.,
  Amer. Math. Soc., Providence, RI, 2016, pp.~207--217.

\bibitem{cauchytransform}
{\sc J.~Cima, A.~Matheson, and W.~T. Ross}, {\em The {C}auchy transform},
  vol.~125 of Mathematical Surveys and Monographs, American Mathematical
  Society, Providence, RI, 2006.

\bibitem{constr}
{\sc O.~El-Fallah, E.~Fricain, K.~Kellay, J.~Mashreghi, and T.~Ransford}, {\em
  Constructive approximation in de {B}ranges-{R}ovnyak spaces}, Constr.
  Approx., 44 (2016), pp.~269--281.

\bibitem{havinbook}
{\sc V.~P. Havin and B.~J\"oricke}, {\em The uncertainty principle in harmonic
  analysis}, vol.~72 of Encyclopaedia Math. Sci., Springer, Berlin, 1995.

\bibitem{sarasonbook}
{\sc D.~Sarason}, {\em Sub-{H}ardy {H}ilbert spaces in the unit disk}, vol.~10
  of University of Arkansas Lecture Notes in the Mathematical Sciences, John
  Wiley \& Sons, Inc., New York, 1994.

\end{thebibliography}

\Addresses

\end{document}